\documentclass[11pt, letterpaper]{amsart}

\usepackage{color}
\usepackage{amsfonts}
\usepackage{amssymb}
\usepackage{amsmath}
\usepackage{hyperref}

\newcommand{\La}{\langle}
\newcommand{\Ra}{\rangle}

\newcommand{\df}{\buildrel\rm{def}\over=}

\newcommand{\Bel}{\mathbf B}

\newcommand{\cz}{Calder\'{o}n--Zygmund }

\newcommand{\pd}{\partial}

\def\supp{\operatorname{supp}}

\newcommand{\ch}{\operatorname{ch}}

\newcommand{\QQ}{[w]_{A_2}}

\newcommand{\eps}{\varepsilon}
\newcommand{\la}{\lambda}
\newcommand{\vf}{\varphi}

\newcommand{\Om}{\Omega}

\newcommand{\cB}{\mathcal B}
\newcommand{\cD}{\mathcal D}

\newcommand{\cI}{\mathcal I}

\newcommand{\cL}{\mathcal L}

\newcommand{\cS}{\mathcal S}

\newcommand{\bR}{\mathbb R}

\newcommand{\mfA}{\mathfrak A}

\def\sp{\operatorname{sp}}

\newcommand{\1}{\mathbf{1}}

\newtheorem{theorem}{Theorem}[section]

\newtheorem{lemma}[theorem]{Lemma}

\theoremstyle{definition}

\numberwithin{equation}{section}

\newcounter{vremennyj}

\begin{document}
\title[Strong and restricted weak weighted  square function estimates]%
{
On strong weighted and restricted weak weighted estimates of the square function
}
\author[P.~Ivanisvili]{Paata Ivanisvili}
\address{Department of Mathematics, Michigan Sate University, East Lansing, MI. 48823}
\email{ivanisvi@math.msu.edu \textrm{(P. \ Ivanisvili)}}
\author[P. Mozolyako]{Pavel Mozolyako}
\thanks{PM is supported by the Russian Science Foundation grant 17-11-01064}
\address{Universit\`{a} di Bologna, Department of Mathematics, Piazza di Porta S. Donato, 40126 Bologna (BO)}
\email{pavel.mozolyako@unibo.it}
\author[A.~Volberg]{Alexander Volberg}
\thanks{AV is partially supported by the NSF grant DMS-1265549 and by the Hausdorff Institute for Mathematics, Bonn, Germany}
\address{Department of Mathematics, Michigan Sate University, East Lansing, MI. 48823}
\email{volberg@math.msu.edu \textrm{(A.\ Volberg)}}
\makeatletter
\@namedef{subjclassname@2010}{
  \textup{2010} Mathematics Subject Classification}
\makeatother
%
%
\begin{abstract}
In this note  we give a sharp weighted estimate for square function  from $L^2(w)$ to $L^2(w)$, $w\in A_2$. This has been known. We also give sharp weighted estimate {\it on test functions} for square function  from $L^2(w)$ to $L^{2,\infty}(w)$, $w\in A_2$. This has not been known.
We also reduce the full estimate for square function  from $L^2(w)$ to $L^{2,\infty}(w)$, $w\in A_2$ to a the estimate of the solution of a certain partial differential inequality.
 
\end{abstract}
\maketitle


\section{Introduction}
\label{intro}

In this note  we give a sharp weighted estimate for square function  from $L^2(w)$ to $L^2(w)$, $w\in A_2$. This has been known. We also give sharp weighted estimate {\it on test functions} for square function  from $L^2(w)$ to $L^{2,\infty}(w)$, $w\in A_2$. This has not been known.

We also reduce the full estimate for square function  from $L^2(w)$ to $L^{2,\infty}(w)$, $w\in A_2$ to a the estimate of the solution of a certain partial differential inequality.

\section{Sharp weighted estimate $S: L^2(w)\to L^2(w)$. The lego construction.}
\label{strongS}

We use the approach quite different from the approach in \cite{HuTV}, where such sharp estimate was first obtained. Notice that more general sharp weighted  estimates for square function in $L^p(w)$ were obtained since then by Lerner
\cite{Le}.

\begin{equation}
\label{strongT}
\frac{1}{|J|} \sum_{I\in D(J)} |\Delta_I w^{-1}|^2 \La w\Ra_I |I|\le C_{s,T}(\QQ) \La w^{-1}\Ra_J\,.
\end{equation}

\begin{equation}
\label{strong}
\frac{1}{|J|}\sum_{I\in D(J)} |\Delta_I (\vf w^{-1})|^2 \La w\Ra_I  |I| \le C_s(\QQ) \La \vf^2w^{-1}\Ra_J\,.
\end{equation}

{\bf Remark.}
We are working with operator $S: L^2(w)\to L^{2,\infty}(w)$ or $S: L^2(w)\to L^{2}(w)$. However, it is more convenient to work with isomorphic objects: $S_{w^{-1}}: L^2(w^{-1})\to L^{2,\infty}(w)$ or $S_{w^{-1}}: L^2(w^{-1})\to L^{2}(w)$, here $S_{w^{-1}}$ denotes the product $SM_{w^{-1}}$, where $M_{w^{-1}}$ is the operator of multiplication.

\bigskip

\subsection{Sharp  estimate for $C_{s, T}(\QQ)$.}
\label{CsT}

{\bf Definition.}
\label{d2}
For a smooth function $B$ of $d$ real variables $(x_1, \dots, x_d)$ we denote by $d^2 B(x)$ the second differential form of $B$, namely,
$$
d^2B(x) = (H_B(x)  dx, dx)_{\bR^d}\,,
$$
where vector $dx= (dx_1, \dots, dx_d)$ is an arbitrary vector in $\bR^d$, and $H_B(x)$ is the $d\times d$  matrix of the second derivatives of $B$ (Hessian matrix) at point $x\in \bR^d$.

For brevity we write $A_2$ in this Section, but we mean the dyadic class $A_2$.

\begin{theorem}
\label{t-strongT}
$C_{s,T} \le A \QQ^2 $ and this estimate is sharp.
\end{theorem}

We introduce the following function of $2$ real variables
\begin{equation}
\label{BelTs}
\Bel(u, v) :=\Bel_Q(u, v) := \sup \frac{1}{|J|} \sum_{I\in D(J)} |\Delta_I w^{-1}|^2 \La w\Ra_I |I| \,,
\end{equation}
where supremum is taken over all $w\in A_2, \QQ\le Q$, such that 
$$
\La w\Ra_J =u, \La w^{-1}\Ra_J = v\,.
$$
Notice that by scaling argument our function does not depend on $J$ but depends on $Q=[w]_{A_2}$. Notice also that function $\Bel_Q$ is 
defined in the domain $O_Q$, where
$$
O_Q:=\{ (u, v) \in \bR^2: 1<uv \le Q\}\,.
$$
Function $\Bel_Q$  is the Bellman function of our problem. In particular, it is very easy to observe that
to prove the estimate in Theorem \ref{t-strongT}  is equivalent to proving $\Bel_Q(u, v) \le A\,Q^2$, and the sharpness in Theorem \ref{t-strongT} is just the claim that
$\sup_{(u, v)\in O_Q} \frac1{v} B_Q(u, v) \ge c>0$.

{\bf Remark.}
Therefore, Theorem  \ref{t-strongT} can be proved by finding the explicit formula for $\Bel_Q$. To do that we obviously need to solve 
an infinite dimensional optimization problem of finding the (almost) best possible $w\in A_2$ (the reader is recalled that it is a dyadic class) such that 
$
\La w\Ra_J =u, \La w^{-1}\Ra_J = v\,.
$
This can be done, and we give this formula in the Addendum. But here we adapt a slightly different approach to proving Theorem \ref{t-strongT}.

\begin{proof}
Below, $A, a$ are positive absolute constants. Instead of finding precisely $\Bel_Q$, we will find another function $B_Q$ such that the following properties are satisfied:

\begin{itemize}
\item $B_Q$ is defined in $O_{4Q}$,
\item $0\le B_Q(u, v) \le A\,Q^2 v$,
\item $B_Q(u, v)-\frac{B_Q(u_+, v_+)+ B_Q(u_-, v_-)}{2} \ge a(v_+ - v_-)^2 u$, if $x=(u,v) \in O_Q, x_\pm = (u_\pm,v_\pm)\in O_Q$, and $x=\frac{x_+ + x_-}{2}$
\end{itemize}

For example, it is not difficult to check that $\Bel_{4Q}$ satisfies all the first and the third properties. We leave it as an exercise to the reader. However, as we said, the second property is not easy, it can be observed when the complicated formula for $\Bel$ is written down (see Addendum). Here we will write down an explicit (and rather easy) form of {\it some} $B_Q$ that satisfies all three properties.

Firs let us observe that if the existence of such a $B_Q$ is proved, the inequality in Theorem \ref{t-strongT} gets proved.
In fact, fix $I\in D(J)$, and introduce $x_I=(\La w\Ra_I, \La w^{-1}\Ra_I)$, $x_{I_\pm}=(\La w\Ra_{I_\pm}, \La w^{-1}\Ra_{I_\pm})$. Of course $\{x_I\}_{I\in D(J)}$ is a martingale. Now we compose this martingale with $B_Q$ (notice that $x_I\in O_Q$, so $B_Q(x_I)$ is well-defined). The resulting object is not a martingale anymore, but it is a super-martingale, moreover, by the third property
$$
\La w\Ra_I |\Delta_I w^{-1}|^2\,|I| \le |I|B_Q(x_I) - |I_+|B_Q(x_{I_+})- |I_-| B_Q(x_{I_-})\,.
$$

Now the reader knows what happens next: we use the telescopic nature of the sum in the right hand side to observe that the summation
in all $I\in D(J)$ cancels all the terms except $|IJB_Q(x_J) $, which, by the second property is at most $A\,Q^2\La w^{-1}\Ra_J |J|$. Hence, we obtained
$$
\sum_{I\in D(J)} |\Delta_I w^{-1}|^2 \La w\Ra_I |I| \le A\,Q^2\La w^{-1}\Ra_J |J|\,,
$$
We leave as an exercise for the reader to explain, where we used the positivity of $B$ in this reasoning.

The last estimate is precisely inequality \eqref{strongT} and Theorem \ref{t-strongT} gets 
proved (apart from the sharpness) as soon as any function $B_Q$ as above is proved to exist.

The construction of  a certain $B_Q$ with above mentioned three properties is split to two steps.

\bigskip

\noindent{\bf Step 1: the reduction to non-linear ODE.} 

\medskip

First of all we wish to find a smooth $B(u, v)$ in the domain
$$
O:= O_Q:=\{ (u,v)>0: 1\le uv \le Q\}\,,
$$
such that the following quadratic forms inequality holds in $O_Q$:
\begin{equation}
\label{sT-infen}
-\frac12 d^2B \ge u(dv)^2 \,.
\end{equation}
We will be searching for homogeneous $B$: $B(u/t, tv) =tB(u, v)$. Hence
$$
B(u,v) = \frac1u \phi(uv)\,
$$
 Then \eqref{sT-infen} becomes
$$
\begin{bmatrix}
x^2 \phi''(x) - 2x\phi'(x) + 2\phi\,, & x\phi''(x)\\
x\phi''(x)\,, & \phi''(x) +2
\end{bmatrix}\le 0\,.
$$
To have this it is enough to satisfy for all $x\in [1, Q]$
\begin{equation}
\label{sT-infenM}
\begin{aligned}
& \phi''(x)+2 \le 0,\, -x\phi'(x) + \phi(x) \le 0,\, 
\\
& \phi''\cdot (-x\phi'+\phi) + x^2 \phi'' -2x\phi' +2\phi =0\,.
\end{aligned}
\end{equation}

The last equation is just making the determinant of our matrix to vanish. Let us start with this equation and put $g=\phi(x)/x$. Then we know that $-x^2g'= \phi-x\phi'\le 0$, so $g$ is increasing.
Also $x g'' + 2g' =\phi'' \le -2$, hence $g'' \le 0$ as $g$ was noticed to be increasing.

 In terms of $g$ we have equation
$$
x(-g'g'' +g'') - 2(g')^2=0\,.
$$
This is a first order non-linear ODE on $h:= g'$ of which we know that $h\ge 0, h'\le 0$:
$$
x(-hh' +h') - 2h^2=0\,.
$$
Variables separate and we get
\begin{align}\label{lambert}
\frac{1-h}{h^2} h' = \frac2{x}\,.
\end{align}
The condition we saws that $h'=g''$ is negative and $x$ here is positive, so $h\geq 1$, and the condition $\phi-x\phi'\leq 0$ is the same as $h\geq0$. Thus any solution $h\geq 1$ of  (\ref{lambert}) gives the desired result.

We want to solve this for $x\in [1, Q]$:
\begin{align*}
-\log h - \frac1h = 2 \log x +c \quad \Leftrightarrow \quad -\frac{1}{h}e^{-\frac{1}{h}}=-x^{2}C, \quad C>0.
\end{align*}

Notice that Lambert $W$ function (which is multivalued) solves the equation $z=W(z)e^{W(z)}$. Thus we must have $W(-x^{2}C)=-\frac{1}{h(x)}$. The condition $-1\leq  -\frac{1}{h(x)}\leq  0$ requires that $-1\leq W \leq 0$, and this gives single-valued function $W_{0}(y)$ defined on the interval $[-1/e, 0]$ such that $W_{0}(-1/e)=-1$, $W_{0}(0)=0$ and $W_{0}(y)$ is increasing. So $h(x)=-\frac{1}{W_{0}(-x^{2}C)}$. The condition $-1/e\leq-x^{2}C\leq 0$ for $x\in [1,Q]$ gives the range for constant $C$ i.e.,  $ 0<C\leq \frac{1}{Q^{2}e}$.  Going back to the functions $\phi$ and $B$ we obtain:
\begin{align*}
& \varphi(x)=-x\int_{1}^{x} \frac{dt}{W_{0}(-t^{2}C)} +x\varphi(1) 
\\
& \text{and thus}\quad B(u,v)=-v\int_{1}^{uv} \frac{dt}{W_{0}(-t^{2}C)} +v\varphi(1).
\end{align*}

Let us also see how bounded is $B$. Choosing $C=\frac{1}{Q^{2}e}$ gives minimal $B$.   $B(u,v)$   can be assumed to be $0$ if $uv=1$. Hence $\phi(1)=0$. Then
\begin{equation}
\label{grow-strongT}
\begin{aligned}
& B^Q(u, v):=B(u,v) =-v\int_{1}^{uv}\frac{dt}{W_{0}(-\frac{t^{2}}{Q^{2}e})} \leq 
\\
& v\int_{1}^{uv}\frac{Q^{2}e}{t^{2}}=e Q^{2}v\left(1-\frac{1}{uv} \right), 1\le uv \le Q\,.
\end{aligned}
\end{equation}
Here we used the fact that Lambert function $W_{0}(x)\leq x$ for $x\in [-1/e,0]$. Actually one can get better estimates by using the series expansion for $W_{0}$ i.e., 
$$
W_{0}(x) = \sum_{n=1}^{\infty}\frac{(-n)^{n-1}}{n!}x^{n}, \quad |x| < \frac{1}{e}. 
$$

\bigskip

\noindent{\bf Step 2: from infinitesimal inequality on $d^2B^Q$ to  global concavity property of $B^Q$. }

\medskip

 The function $B^Q$ defined in \eqref{grow-strongT}  is not function $B_Q$ with three properties formulated at the beginning of the proof of this theorem.
However, let us prove that $B_Q:=B^{4Q}$ has all these three properties. The first property is just by definition, and the second property is because we just proved in \eqref{grow-strongT} that 
$$
B_Q(u, v)=B^{4Q}(u, v) \le 4e Q^2 v\,.
$$

To prove the third property let us fix $x=(u,v)\in O_Q$,  $x_\pm = (u_\pm, v_\pm)\in O_Q$ such that $x=\frac{x_+ + x_-}{2}$. Introduce two functions defined on $[-1,1]$:
$$
U(t) = \frac{1+t}2 u_+ + \frac{1-t}2 u_-,\,\,V(t) = \frac{1+t}2 v_+ + \frac{1-t}2 v_-\,.
$$
Compose the vector function $X(t)=(U(t), V(t)$ and function $B_Q=B^{4Q}$, namely, put
$$
b(t):= B_Q(X(t))\,.
$$
Then $X(\pm 1) = x_\pm, X(0)= x$. It is important to notice that $b$ is well-defined because
$$
\forall t\in [-1,1]\,\,\,X(t)\in O_{4Q}\,.
$$
The latter is an elementary geometric observation saying that if three points $X_\pm, X$ belong to $O_Q$, and $X=\frac{X_+ +X_-}{2}$, then the whole segment with end points $X_\pm$ lie in $O_{4Q}$. (But not, in general, on $O_{Q'}$ with $Q'< 4Q$.)
Now we differentiate twice function $b$. The chain rule gives us immediately
$$
b''(t) = (H_{B_Q}(X(t)) (x_+ - x_-), (x_+ - x_-))_{\bR^2}\,,
$$
where $H_{B_Q}$ denotes as always the Hessian matrix of function $B_Q$. Therefore, the use of \eqref{sT-infen} gives us
$$
-b''(t) \ge 2 U(t) (v_+ - v_-)^2\,.
$$
On the other hand,
$$
B_Q(x)-\frac{B_Q(x_+)+ B_Q(x_-)}{2} = b(0)- \frac{b(1)+ b(-1)}{2} = \frac12\!\!\int_{-1}^1\!\!\!\!\! -b''(t) (1-|t|)\, dt
$$
Notice that the integrand is always nonnegative by the previous display formula. By the same formula, the integrand is at least $u=U(0)$ for $t\in [0, 1/2]$ because on this interval $U(t)\ge \frac12 U(0)$ by the obvious geometric reason. Combing that we obtain
$$
B_Q(x)-\frac{B_Q(x_+)+ B_Q(x_-)}{2} \ge \frac38 u(v_+ - v_-)^2\,.
$$
We established all three properties for $B_Q$, and we have already shown that this is enough to prove the inequality in Theorem \ref{t-strongT}.
The sharpness is not difficult to see for a weight with one singular point, see \cite{HuTV} for example.
\end{proof}

\bigskip

\subsection{Proving the instance of $T1$ theorem using its Bellman function.}
\label{T1Bel}

\begin{theorem}
\label{t-strong}
$C_{s} \le A Q^2$ and this estimate is sharp.
\end{theorem}

Let us deduce this result from Theorem \ref{t-strongT}.
This the occasion of the so-called weighted $T1$ theorem.

We use the notation $h_I$ for a standard Haar function supported on dyadic interval $I$, it is given by
$$
h_I=\begin{cases}
\frac1{\sqrt{I}},\, \text{on}\,\, I+\\
-\frac1{\sqrt{I}},\, \text{on}\,\, I-\,.
\end{cases}
$$
It is an orthonormal basis in unweighted $L^2$. Now consider the same type of Haar basis but in weighted $L^2(w^{-1})$: functions $h_I^{w^{-1}}$ are orthogonal to constants in $L^2(w^{-1})$, normalized in $L^2(w^{-1})$, assume constant value on each child of $I$, and are supported on $I$. For dyadic intervals on the line we get

We need a couple of lemmas.
\begin{lemma}
\label{haarF}
The following holds $h_I = \alpha_I^{w^{-1}} h_I^{w^{-1}} + \beta_I^{w^{-1}} \frac{\1_I}{\sqrt I}$ with
\begin{equation}
\label{albe}
\alpha_I^{w^{-1}} =\frac{\La w^{-1}\Ra^{1/2}_{I+}\La w^{-1}\Ra^{1/2}_{I-}}{\La w^{-1}\Ra^{1/2}_{I}}\,,\, \beta_I^{w^{-1}}  = \frac{\La w^{-1}\Ra_{I+}-\La w^{-1}\Ra_{I-}}{\La w^{-1}\Ra_I}
\end{equation}
\end{lemma}

\begin{proof}
This is a direct calculation, we need to define two constants $\alpha_I^{w^{-1}},  \beta_I^{w^{-1}}$ and we have two conditions: $\|h_I^{w^{-1}}\|_{L^2(w^{-1})} =1$ and $(h_I^{w^{-1}}, \1)_{L^2(w^{-1})} =0$. 
\end{proof}

The second lemma is just the instance of the chain rule. 
\begin{lemma}
\label{chaind2}
Let $\Phi(x'), \!B(x'')$  be smooth functions of $x'\!=\!\!(x_1,\!\dots, \!x_n,\!x_0)$, $x''=(x_{n+1},\dots, x_m)$. Then we compute the second differential form of the composition function 
$$
\cB(x_1,\dots,x_n, x_{n+1}, \dots, x_m)= \Phi(x_1,\dots, x_n, B(x_{n+1}, \dots, x_m))
$$
 by the following formula:
 $$
 d^2 \cB = d^2\Phi + \frac{\pd \Phi}{\pd x_0} d^2 B\,.
 $$
 \end{lemma}
 
 {\bf Remark.}
 We understand the left hand side as $(H_{\cB}(x) dx, dx)_{\bR^m}$, where
 $$
 x:=(x_1,\dots,x_n, x_{n+1}, \dots, x_m),\,\, dx= (dx_1,\dots, dx_n, dx_{n+1}, \dots, dx_m)\,.
 $$
 We understand $d^2\Phi$ in the right hand side as $(H_{\Phi} (x_1, \dots, x_n, B(x'')) dy, dy)_{\bR^{n+1}}$, where
 $$
 dy= (dx_1, \dots, dx_n, dB),\,\, dB= (\nabla B(x''), dx'')_{\bR^{m-n}}\,.
 $$

 Now we are ready to prove Theorem \ref{t-strong}.
 
 \begin{proof}
 The sum we want to estimate in \eqref{strong}
$$
\frac1{|J|}\sum_{I\in D(J)} (\La \vf w^{-1}\Ra_{I+}- \La \vf w^{-1}\Ra_{I-})^2\La w\Ra_I |I|
$$
is of course
$$
\Sigma=  \frac2{|J|}\sum_{I\in D(J)} (\vf w^{-1}, h_I)^2 \La w\Ra_I.
$$
We can plug the decomposition of Lemma \ref{haarF} and take into account that for dyadic lattice  obviously
$\alpha_I^{w^{-1}} \le 2 \La w^{-1}\Ra^{1/2}_I$. Then we obtain
\begin{align*}
& \Sigma \le \frac8{|J|}\sum_{I\in D(J)} (\vf w^{-1}, h_I^{w^{-1}})^2\La w^{-1}\Ra_I \La w\Ra_I +
\\
& \frac2{|J|}\sum_{I\in D(J)}\La \vf w^{-1}\Ra_I^2 (\beta_I^{w^{-1}})^2 \Ra_I \La w\Ra_I |I|=: \Sigma_1 + \Sigma_2\,.
\end{align*}
The system $\{h_I^{w^{-1}}\}_{I\in D(J)}$ is orthonormal in $L^2(w^{-1})$, and $\La w^{-1}\Ra_I \La w\Ra_I \le Q$. Hence, immediately we have
\begin{equation}
\label{Sigma1}
\Sigma_1 \le 8Q \|f\|^2_{L^2(w^{-1})}\,.
\end{equation}
We are left to estimate $\Sigma_2$. 
To do that let us rewrite $\Sigma_2$:
$$
\Sigma_2= \frac2{|J|}\sum_{I\in D(J)}\left(\frac{\La \vf w^{-1}\Ra_I}{\La w^{-1}\Ra_I}\right)^2 \gamma_I |I|=\frac2{|J|}\sum_{I\in D(J)}\left(\La \vf w^{-1}\Ra_{I, w^{-1}}\right)^2 \gamma_I |I| \,,
$$
where $\La \cdot\Ra_{I, w^{-1}}$ means the average with respect to measure $\mu:=w^{-1}(x)dx$ and
$$
\gamma_I:= \left(\La w^{-1}\Ra_{I+}-\La w^{-1}\Ra_{I-}\right)^2\La w\Ra_I\,.
$$

We are going to prove now that with some absolute constant $A$
\begin{equation}
\label{embSq}
\frac1{|J|}\sum_{I\in D(J)}\left(\La \vf w^{-1}\Ra_{I, w^{-1}}\right)^2 \gamma_I |I| \le A Q^2\La \vf^2 w^{-1}\Ra_J\,.
\end{equation}
This of course finishes the proof of Theorem \ref{t-strong}.

To prove \ref{embSq} we will construct a special function of $4$ real variables $\cB (X)=\cB_Q(X), X=(F, f, u, v),$ that possesses the following properties

\begin{itemize}
\item 1) $\cB_Q$ is defined in a non-convex domain $O_{4Q}$, where $O_Q:=\{(F, f, u, v)\in \bR_+^4: f^2 < Fv, \, 1<uv <Q\}$,
\item 2)  $0 \le \cB_Q \le F$,
\item 3)  $\cB_Q (X) -\frac{\cB_Q(X_+)+ \cB_Q(X_-)}{2} \ge aQ^{-2}\frac{f^2}{v^2} u (v_+ - v_-)^2$, where $X=(F, f, u, v), X_\pm =(F_\pm, f_\pm, u_\pm, v_\pm)$ belong to $O_Q$, and $X=\frac{X_+ +X_-}{2}$.
 \end{itemize}

As soon as such a function constructed \eqref{embSq} and Theorem \ref{t-strong} follow immediately. In fact we repeat our telescopic consideration. We set the vector martingale 
$X_I:= ( F_I, f_I, u_I, v_I)$, where
$$
F_I = \La \vf w^{-1}\Ra_I,\, f_I =\La \vf w^{-1}\Ra_I, u_I=\La w\Ra_I, \, v_I=\La w^{-1}\Ra_I\,.
$$
It is obvious that vector martingale $\{X_I\}_{I\in D(J)}$ is always inside $O_Q$, and 
so the superposition of this martingale and $\cB_Q$ is well defined: $\cB_Q(X_I)$. 
Then, the property 3) claims that $\{\cB_Q(X_I)\}_{I\in D(J)}$ is a super-martingale, and moreover
$$
|I| \frac{f_I^2}{v_I^2} u_I (v_{I_+} - v_{I_-})^2 \le AQ^2 ( |I|\cB(X_I) - |I_+|\cB(X_{I_+}) - |I_+|\cB(X_{I_-}))
$$
We use the telescopic nature of the term in the right hand side, and summing these terms for all $I\in D(J)$, we then notice that all of them will cancel each other, except
$AQ^2  |J|\cB(X_J)$, which is bounded by $AQ^2 |J|F_J= AQ^2|J| \La \vf^2 w^{-1}\Ra_J$. We proved \eqref{embSq} provided that the existence of $\cB_Q$ is validated.

Now we will write the explicit formula for $\cB_Q$. Exactly as in Theorem \ref{t-strongT} we first construct, by an explicit formula, an auxiliary function $\cB^Q$.
Here it is
\begin{equation}
\label{cBQ}
\cB^Q(F, f, u, v) := F- \frac{f^2}{ v+ aQ^{-2} B^Q(u, v)}\,,
\end{equation}
where $B^Q$ was defined in \eqref{grow-strongT}. It is clear that it satisfies property 2).  
It  ``almost" satisfies property 1), but it is defined only in $O_Q$, not in a larger domain $O_{4Q}$. 
As to the property 3) it does satisfy its infinitesimal version: at point $X=(F, f, u, v)\in O_Q$
\begin{equation}
\label{BQ-diff}
-d^2B^Q \ge aQ^{-2} \frac{f^2}{v} u (dv)^2\,.
\end{equation}
Let us prove this.
This follows from Lemma \ref{chaind2}. In fact,consider $\Phi(x_1, x_2, x_3, x_0):= x_1 -\frac{x_2^2}{x_3+x_0}$. 
By a direct simple calculation one can see that it is concave in $\bR_+^4$, so $d^2\Phi\le 0$. Now we see that 
$$
\cB^Q (F, f, u, v) = \Phi(F, f, v, aQ^{-2} B^Q(u, v))\,,
$$
and by Lemma \ref{chaind2}
$$
d^2\cB^Q = d^2\Phi + aQ^{-2}\frac{\pd \Phi}{\pd x_0} \cdot (d^2 B^Q)\le \frac{a}{Q^2} \frac{f^2}{(v+ aQ^{-2} B^Q(u, v))^2} (d^2 B^Q)\,.
$$
But in Theorem \ref{t-strongT} we proved that $B^Q \le AQ^2 v$, hence, 
choosing small absolute constant $a$ we guarantee that $v+ aQ^{-2} B^Q(u, v) \le 2v$. 
We also proved in Theorem \ref{t-strongT} that $-d^2B^Q \ge u (dv)^2$. Combining these facts with the last display inequality we obtain
$$
-d^2\cB^Q  \ge \frac{a}{4Q^2}\frac{f^2}{v^2}u (dv)^2\,,
$$
which is precisely \eqref{BQ-diff}. 

We need function with property 3) and defined in the domain $O_{4Q}$. So let us put $\cB_Q:= \cB^{4Q}$. 
Exactly as before as in Step 2 of Theorem \ref{t-strongT}  we can prove now that not only infinitesimal 
special concavity holds in the form $-d^2\cB_Q  \ge \frac{a}{64Q^2}\frac{f^2}{v^2}u (dv)^2$, but also  we have with some small positive $a_0$
\begin{equation}
\label{BQ-discr}
\cB_Q (X) -\frac{\cB_Q(X_+)+ \cB_Q(X_-)}{2} \ge a_0Q^{-2}\frac{f^2}{v^2} u (v_+ - v_-)^2\,,
\end{equation}
for all triple of points $(X, X_+, X_-)\in O_Q^3$ such that $X=(F, f, u, v), X_\pm =(F_\pm, f_\pm, u_\pm, v_\pm)$, and  $X=\frac{X_+ +X_-}{2}$.
 \end{proof}

 {\bf Remark.}
 \label{T1principle}
 The constant in the right hand side of the main inequality \eqref{embSq}
 is $AQ^2$.
 It is important to notice  that the constant in the right hand side of \eqref{embSq}  depends only on the constant in Theorem \ref{t-strongT}. 
 
 In fact, these two constants are equal up to an absolute constant. So if, for example, the constant in Theorem \ref{t-strongT} is denoted by $K$, then the constant in \eqref{embSq}  becomes $AK$, and the constant in Theorem \ref{t-strong} would also become $AK$ (with another absolute constant $A$). This principle of passing from uniform  testing conditions (Theorem \ref{t-strongT}) to full boundedness of the operator is quite general, and, as we said already,  is often called $T1$ theorem. 
 
 Here this principle was proved by showing that {\it the formula} that proves Theorem \ref{t-strong} can be obtained by totally formal manipulations from {\it the formula} that prove testing condition (Thorem \ref{t-strongT}). See the next remark.

 {\bf  Remark.}
 \label{lego-r}
 The reader should pay attention to a  following curious formula.
  \begin{equation}
 \label{lego}
 \cB_Q(F, f, u, v) = F-\frac{f^2}{v+ aQ^{-2} B^Q(u, v)}\,,
 \end{equation}
 This is the function $\cB_Q$  that proves Theorem \ref{t-strong}, and
$B^Q$ is the function that proved Theorem \ref{t-strongT}. 

So we see another instance of transference by use of Bellman function. By formula \eqref{lego} we transfer the  claim of Theorem \ref{t-strongT} to Theorem \ref{t-strong}. A Bellman function of Theorem  \ref{t-strongT}  was used as ``a lego piece" to construct  a Bellman function for Theorem \ref{t-strong}. In this instance this ``lego construction" proved for us the $T1$ theorem for the weighted square function operator.

\bigskip

\section{A small sharpening of the $T1$ theorem for dyadic square function}
\label{ning}

Recall that
$$
\gamma_I = (\La w^{-1}\Ra_{I_+} - (\La w^{-1}\Ra_{I_-})^2 \La w\Ra_I\,.
$$

A natural question arises: Let $[w]_{A_2}= Q>>1$ and  at the same time
\begin{equation}
\label{T1cond}
\frac1{|J|}\sum_{I\in \cD(J)} \gamma_I |I| \le q^2 \La w^{-1}\Ra_J,\,\,\, \forall J\in \cD,
\end{equation}
where $q<<Q$, does this mean that the norm $\|S_{w^{-1}}: L^2(w^{-1})\to L^2(w)\|$ is bounded by $Cq$?
That would be quite expected because this statement reminds the statement of the so-called $T1$ theorems.
In fact, \eqref{T1cond} is precisely the testing condition and can be rewritten as
\begin{equation}
\label{T1cond1}
\|S_{w^{-1}}\chi_J\|_w^2 \le q^2 \|\chi_J\|_{w^{-1}}^2,\,\,\, \forall J\in \cD.
\end{equation}
We are quite sure that this is wrong. If so the next natural question is the following: if, however, $q<<Q$, is it true that
one can give a better estimate than the one we proved in the previous section:
$$
\|S_{w^{-1}}: L^2(w^{-1})\to L^2(w)\| \lesssim  Q\,?
$$

The answer to this question is positive. We have the following result.
\begin{theorem}
\label{qQ}
Let $w\in A_2$ (dyadic as always), let  $[w]_{A_2}= Q$, and let \eqref{T1cond} holds with $q<<Q$, then we have the following improved estimate on $\|S_{w^{-1}}: L^2(w^{-1})\to L^2(w)\| $:
$$
\|S_{w^{-1}}: L^2(w^{-1})\to L^2(w)\|  \le C\,(Q^{1/2} + q)\,.
$$
\end{theorem}

\begin{proof}
Let us consider the following aggregate, which is quite akin to the function from \eqref{cBQ} ( constant $a$ is again a small positive constant):
\begin{equation}
\label{cBQA}
\cB(F, f, A, v) := F- \frac{f^2}{ v+ aq^{-2} A}\,,
\end{equation}
Given $I\in \cD$, we put
$$
A_I := \frac1{|I|} \sum_{\ell\in \cD(I)} \gamma_\ell |\ell|\,.
$$
As before we set
$Y_I:= ( F_I, f_I, A_I, v_I)$, where is defined above, and 
$$
F_I = \La \vf w^{-1}\Ra_I,\, f_I =\La \vf w^{-1}\Ra_I,  v_I=\La w^{-1}\Ra_I\,.
$$
We also denote $y_I := ( F_I, f_I, \frac12( A_{I_+}+A_{I_-}), v_I)$.
We can estimate
\begin{align*}
 &|I|\cB(Y_I) - |I_+|\cB(Y_{I_+}) - |I_+|\cB(Y_{I_-}))
 \\
&  = |I| \left((\cB(Y_I) - \cB(y_I )) + (\cB(y_I) -  \frac12\cB(Y_{I_+}) - \frac12\cB(Y_{I_-}))\right)
\\
 & \ge  \frac{|I|}{4} \frac{f_I^2}{v_I^2} \big(A_I - \frac12( A_{I_+}+A_{I_-})\big) \frac{a}{q^2}=  \frac{a}{4q^2} \frac{f_I^2}{v_I^2} \gamma_I |I|
 \\
 &= \frac{a}{4q^2}\frac{ \La \vf w^{-1}\Ra_I^2}{\La w^{-1}\Ra_I^2} (\La w^{-1}\Ra_{I_+} - (\La w^{-1}\Ra_{I_-})^2 \La w\Ra_I \,|I|
\end{align*}
Summing up over all $I\in \cD(J)$ and using the telescopic nature of terms, we get
\begin{equation}
\label{qembSq}
\frac1{|J|}\sum_{I\in D(J)}\left(\La \vf w^{-1}\Ra_{I, w^{-1}}\right)^2 \gamma_I |I| \le A q^2\La \vf^2 w^{-1}\Ra_J\,,
\end{equation}
which is like \eqref{embSq}, but with $q$ replacing $Q$. In the inequality above we used the concavity of function $\cB(F, f, A, v) = F- \frac{f^2}{ v+ a\,q^{-2} A}$, which, in particular, gives us that 
$$
 \cB(y_I) -  \frac12\cB(Y_{I_+}) - \frac12\cB(Y_{I_-}) \ge 0\,,
 $$
 and we use the estimate from below for $\frac{\pd}{\pd A}\cB(F, f, A, v)$:
 $$
 \min_{\frac12( A_{I_+}+A_{I_-}) \le A\le A_I}\frac{\pd}{\pd A}\cB(F, f, A, v) \ge \frac{f^2}{(v +a\,q^{-2} A_I)^2} \ge  \frac{f^2}{4v^2}\,.
 $$
 The last inequality is clear if we use \eqref{T1cond}.

Now combining estimate  \eqref{Sigma1} of the previous section and \eqref{qembSq} we get the claim of Theorem \ref{qQ}.

\end{proof}

%

\section{Weak weighted estimates for the square function}
\label{w-square}

Our measure space  will be $(X, \mfA, dx)$, where $\sigma$-algebra $\mfA$ is generated by a standard dyadic filtration $\cD=\cup_k \cD_k$ on $\bR$.
We considered the weak weighted estimate for the martingale transform. In \cite{NRVV} the end-point exponent was $p=1$, and critical weights belong to $A_1$ class. 

Now we are  going to consider weighted weak estimates for the dyadic  square function.  The end-point exponent is now $p=2$, and critical weights belong to $A_2$ class. See the estimates for subcritical and supercritical exponents in \cite{LaSc} and \cite{DSaLaRey}.

The weak estimate of the square function in $L^2(w dx)$, $w\in A_1$ is well known, see, e. g. \cite{Wilson}. The fact that the end-point exponent is now at $p=2$ is explainable by the bi-linear nature of square function transform.

Recall that the symbol $\ch(J)$ denotes the dyadic children of $J$. Recall that the martingale difference operator $\Delta_J$ was defined as follows:
$$
\Delta_Jf := \sum _{I\in ch(J)} \1_I ( \La f\Ra_I -\La f\Ra_J).
$$

For our case of dyadic lattice on the line we have  that $|\Delta_Jf|$ is constant on $J$, and
$$
\Delta_Jf =\frac{1}{2}[(\La f \Ra_{J_+}-\La f \Ra_{J_-})\1_{J_+}  + (\La f \Ra_{J_-}-\La f \Ra_{J_+})\1_{J_-}]\,.
$$

The square function  operator is 
$$
(Sf)^2 (x)= \sum_{J\in \cD} |\Delta_J f|^2\1_J(x)\,.
$$

In this Section we work only with the dyadic $A_2$ classes of weights, but we skip the word dyadic, because we consider here only dyadic operators.
We consider a positive  function $w(x)$, and as before we call it  $A_2$ weight if

\begin{equation}
\label{a2}
Q:=\QQ:=\sup_{J\in \cD}\La w\Ra_{J} \La w^{-1}\Ra_J<\infty\,.
\end{equation}

 We are  going to consider restricted weak estimate, when the operator is applied to $f= \1_E$, but only for set $E$ being itself  a dyadic interval.  
 \begin{equation}
\label{weakTSqF}
\frac{1}{|I|} w\left\{x\in I: \sum_{J\in D(I)} |\Delta_J w^{-1}|^2 \1_J(x) >\lambda\right\} \le C_{T}(\QQ) \frac{\La w^{-1}\Ra_I}{\la}\,.
\end{equation}

We give some information on  full weak estimate for the square function operator:

\begin{equation}
\label{weakSqF}
\frac{1}{|I|} w\left\{x\in I: \sum_{J\in D(I)} |\Delta_J( \vf w^{-1})|^2 \1_J(x) >\lambda\right\} \le C(\QQ) \frac{\La \vf^2w^{-1}\Ra_I}{\la}\,.
\end{equation}

Here $\vf$ runs over all functions  such that $\supp \vf\subset I$ and $\vf\in L^2(I, w\, dx), w\in A_2$.

We wish to understand the sharp order of magnitude of the weak estimates constants constants $C_{T}(\QQ), C(\QQ)$  from \eqref{weakTSqF}, \eqref{weakTSqF} correspondingly ($\QQ$ is supposed to be large). We wish to compare them to the similar estimates of strong type: \eqref{strongT}, \eqref{strong} from Section \ref{strongS}.

Constant $C(\QQ)$ gives the dependence on $\QQ$ of the norm of the  sub linear operator $S\vf:= \left(\sum_{J\in \cD} |\Delta_J \vf|^2\1_J(x)\right)^{1/2}$ from $L^2(w)$ to $L^{2,\infty}(w)$. Constant $C_{T}(\QQ)$ gives the dependence on $\QQ$ of the same weak norm but measured only on test functions $\vf$ that are just characteristic functions of dyadic intervals: $\vf= \1_I, I\in \cD$.

\bigskip

Let $M_a$ be an operator of multiplication by function $a$. Notice a convenient change of variable $\vf\to \vf W^{-1}$, which reduces the problem of estimating the square function operator $S: L^2(w)\to L^{2, \infty}(w)$ to the estimate of $S_{w^{-1}}:= S M_{w^{-1}}$ from $L^2(w^{-1})$ to $L^{2,\infty}(w)$.

\bigskip

Of course, there are several obvious estimates: 1) weak constants are smaller than strong constants:
$$
C_{T}(\QQ) \le C_{s, T}(\QQ),\, C(\QQ) \le C_s(\QQ)
$$
just by Chebyshev inequality; also 2) test function estimates are  trivially at least as good as the full estimates:
$$
C_{ T}(\QQ) \le C(\QQ),\, C_{s, T}(\QQ) \le C_s(\QQ)\,.
$$
It is quite well known (although far from being trivial) that there is at least one converse inequality:
\begin{equation}
\label{T1}
C_s(\QQ)\le A(d) C_{s, T}(\QQ)\,.
\end{equation}
This is an instance of the celebrated $T1$ theorem, on this occasion applied to square function transform in the weighted situation. We have proved this statement in Theorem \ref{t-strong} above, see Remark \ref{T1principle}.

\bigskip

{\bf  Remark.}
\label{smaller}
Weak type estimates with sharp dependence on $\QQ$ are rather difficult, in part because their sharp constants are much smaller than the sharp constants in the corresponding strong type estimates.

{\bf  Remark.}
This is in a big contrast with the usual singular integrals $T$ of \cz \, type. We know that for $w\in A_2$ the strong norm $\|T: L^2(w)\to L^2(w)\|$ is exactly equivalent to
$\|T: L^2(w)\to L^{2,\infty}(w)\|+\|T^*: L^2(w^{-1})\to L^{2,\infty}(w^{-1})\|$, see \cite{PTV}, \cite{HPTV} for example.

In view of Remark \ref{smaller} it was  important to review the strong sharp weighted estimate $S: L^2(w)\to L^2(w)$ in Section \ref{strongS}.

{\bf  Remark.}
 The reader can see that \eqref{weakTSqF} estimate is a particular case of \eqref{weakSqF} estimate for a special choice of test function $f=\1_I$. The same remark holds for \eqref{strongT} and \eqref{strong}. We already noted above that test function estimate \eqref{strongT} (plus its symmetric counterpart for $w$ exchanged by $w^{-1}$) imply the strong estimate \eqref{strong}. This, as we already mentioned, is the essence of  weighted $T1$ theorem (testing condition theorem in the terminology of E. Sawyer).

{\bf  Remark.}
Unfortunately there is no $T1$ principle in general for weak type estimates. See the papers \cite{Ca}, \cite{ABMau},  \cite{CaGr}.

\bigskip

\subsection{Sharp constant in weak testing estimate. An ``isoperimetric" problem.}
\label{s-weakT}

\begin{theorem}
\label{t-weakT}
$C_{w,T} \le A Q$ and this estimate is sharp.
\end{theorem}

We will prove the estimate, the sharpness is well known just for one-point-singularity weights.

As always in this Section $A_2$ means  dyadic $A_2$. We introduce the following function of $3$ real variables
\begin{equation}
\label{supTest}
B_Q(u, v, \la):= \sup \frac{1}{|J|} w\left\{x\in J: \sum_{I\in D(J)} |\Delta_I w^{-1}|^2 \chi_I(x) >\lambda\right\} \,,
\end{equation}
where the supremum is taken over all $w\in A_2, [w]_{A_2}\le Q$, such that 
$$
\La w\Ra_J =u, \La w^{-1}\Ra_J = v\,.
$$
Notice that by scaling argument our function does not depend on $J$ but depends on $Q=[w]_{A_2}$. For brevity we can skip $Q$: $B:= B_Q$.

{\bf  Remark.}
Ideally we want to find the formula for this function. Notice that this is similar to solving a problem 
of ``isoperimetric" type, where the solution of certain non-linear PDE is a common tool, see e. g. \cite{BarthH}. 

\bigskip

\subsubsection{Properties of $B$ and the main inequality}
\label{prop-weakT}
 Notice several properties of $B$:
\begin{itemize}
\item
$B$ is defined in  $\Omega:=\{(u,v, \la): 1\le uv \le Q, u>0, v>0, 0\le \la<\infty\}$.
\item If $P=(u,v, \la), P_+=(u_+,v_+, \la_+), P_-=(u_-,v_-,\la_-)$ belong to $\Omega$, and
$u=\frac12(u_+ +u_-)$, $v=\frac12(v_+ +v_-)$, $\la= \min(\la_+, \la_-)$,, then {\bf the main inequality} holds with constant $c=1$:
$$
B\left(u, v, \la+c (v_+ - v_-)^2\right)-\frac{B(P_+)+ B(P_-)}2 \ge 0.
$$
\item $B$ is decreasing in $\la$.
\item Homogeneity: $B(ut, v/t, \la/ t^2) = tB(u,v, \la), t>0$.
\item Obstacle condition: for all points $(u, v, \lambda)$ such that $10\le uv \le Q,\, \la\ge 0$, if $\la \le \delta \,v^2$ for a positive absolute constant $\delta$, one has $B(u, v, \la)= u$.
\item The boundary condition $B(u, v, \la) = 0$ if $uv=1$.
\end{itemize}

All these properties are very simple consequences of the definition of $B$. However, 
let us explain a bit the second and the fifth bullet. 
The second bullet is the consequence of the scale invariance of $B$. 
We consider data $P_+$ and find weight $w_+$ that almost supremizes $B(P_+)$.
By definition of $B$, we  have it on $J$. But by scale invariance we can think that $w_+$ lives on $J_+$. 
Then we consider data $P_-$ and find weight $w_-$ that almost supremizes $B(P_-)$.
Again we are supposed to have it on $J$. But by scale invariance we can think that $w_-$ lives on $J_-$. The next step is to consider the contatination of $w_+$ and $w_-$:
$$
w_c:=\begin{cases} w_+,\,\, \text{on} \,\, J_+\\
w_-,\,\, \text{on} \,\, J_-\,.\end{cases}
$$
Clearly this new weight is a competitor for giving the supremum for date $P$ on $J$. But it is only a competitor, the real supremum in \eqref{supTest} is bigger. This implies the second bullet above ({\it the main inequality}).

\bigskip

Now let us explain the fifth bullet above, we call it {\it the obstacle condition}. 
Let us consider a special weight $w_s$ in $J$: it is one constant on $J_-$ and just another constant on $J_+$. Moreover, we wish to have
$\La w_s^{-1}\Ra_{J_+} = 4 \La w_s^{-1}\Ra_{J_-}$. Notice that then $b\La w_s^{-1}\Ra_J \le |\Delta_J w_s^{-1}|$ 
with some positive absolute constant $b$.  

Now it is obvious that if $\lambda\le \delta^2\La w_s^{-1}\Ra_J^2$ 
then $\{x\in J: S^2_{w_s^{-1}} (\1_J) \ge\lambda\} =J$ and so 
$\frac1{|J|} w_s\{x\in J: S^2_{w_s^{-1}} (\1_J) \ge\lambda\} =\La w_s\Ra_J$.  
Notice now that $w_s$ is just one admissible weight, and that we have to 
take supremum over all such admissible weights. 

We get the fifth bullet above (=the obstacle condition):
$B(u, v, \la)=u$ for those points $(u, v, \la)$ in the domain of definition of $B$, 
where the corresponding $w_s$ with $\La w_s\Ra =u, \La w_s^{-1}\Ra_J= v$ exists. 
It is obvious that for all sufficiently large $Q$ and for any pair $(u,v)$ such that $10\le uv\le Q$ one can construct a just ``two-valued" $w_s$ as above with
$[w_s]_{A_2} \le Q$ (we recall that we deal only with dyadic $A_2$ weights).

\bigskip

Notice that the main inequality above transforms into a {\bf partial differential inequality} if considered infinitesimally (and if we tacitly assume that $B$ is smooth):
\begin{equation}
\label{infMI}
-\frac12 d^2_{u,v}B +c\frac{\partial B}{\partial \la} (dv)^2 \ge 0\,.
\end{equation}
We get it with $c=1$ for the function $B$ defined above  (if $B$ happens to be smooth).

We are not going to find $B$ defined in \eqref{supTest}, but instead we will construct smooth
$\cB$ that satisfies all the properties above (and of course \eqref{infMI}) except for the boundary condition 
(the last bullet above). It will satisfy even a slightly stronger properties, for example, the obstacle condition (the fifth bullet) will be satisfied with $1$ instead of $10$:
\begin{equation}
\label{obstTest}
\begin{aligned}
& \forall (u, v, \lambda)\,\text{ such that}\, 1\le uv \le Q,\, \la\ge 0, 
\\
& \text{if}\,\, \la \le \delta \,v^2\,\,\text{ for some}\, \delta>0, \,\text{then}\,\, B(u, v, \la)= u\,.
\end{aligned}
\end{equation}
Here $a$ will be some positive absolute constant (it will not depend on $Q$).

\bigskip

Using our usual telescopic sums consideration it will be very easy to prove the following

\begin{theorem}
\label{t-weakT}
Suppose we have a smooth function $\cB$ satisfying all the conditions above except the boundary condition, 
but satisfying the obstacle condition in the form \eqref{obstTest}. We also allow $c$ to be a small positive constant (say, $c=\frac18$). And suppose it also satisfies
\begin{equation}
\label{gr-weakT}
\cB(u, v, \la) \le A\,Q\frac{v}{\la}\,.
\end{equation}
Then the constant $C_{w, T}$ in \eqref{weakTSqF} is at most $A\,Q$.
\end{theorem}

\begin{proof}
It is a stopping time reasoning. It is enough to think that $w$ is constant on some very small dyadic intervals and to prove the estimate on $C_{w,T}$ uniformly. Then we start with any such $w, [w]_{A_2}\le Q$, and we use the main inequality with $u=\La w\Ra_J, v=\La w^{-1}\Ra_J$, $u_\pm=\La w\Ra_{J_\pm}, v_\pm=\La w^{-1}\Ra_{J_\pm}$, 
$$ 
\la- c (\La w^{-1}\Ra_{J_+}-\La w^{-1}\Ra_{J_-})^2=: \la_{J_\pm}
$$
to obtain
\begin{equation}
\label{main-wT}
\begin{aligned}
 |J_+| \cB(\La w\Ra_{J_+}, \La w^{-1}\Ra_{J_+}, \la_{J_+})+ |J_-| \cB(\La w\Ra_{J_-}, \La w^{-1}\Ra_{J_-}, \la_{J_-}) \le  
 \\
 \cB(\La w\Ra_J, \La w^{-1}\Ra_J, \la)|J| \le \frac{A\,Q\La w^{-1}\Ra_J}{\la}|J|
\end{aligned}
\end{equation}
We continue to use the main inequality (because $J_\pm$ are not at all different from $J$) 
and finally after large but finite number of steps, on certain  collection $\cI$ of small intervals $I=J_{\pm\pm\cdots\pm}$ we come  to the situation that
\begin{equation}
\label{la-final}
\la_{J_{\pm\pm\cdots\pm}} <c (\La w^{-1}\Ra_{J_{\pm\pm\cdots+}}-\La w^{-1}\Ra_{J_{\pm\pm\cdots {-}}})^2\,.
\end{equation}
Collection $\cI$ may be empty of course, but we know that $I\in \cI$ if 
on $I$   the following holds for $x\in I$: 
$$
c\sum_{L\in D(J), I\subset L} |\Delta_I w^{-1}|^2 \1_L(x)>\lambda\,.
$$
Let us combine \eqref{la-final} with an obvious inequality 
$$
c (\La w^{-1}\Ra_{J_{\pm\pm\cdots+}}-\La w^{-1}\Ra_{J_{\pm\pm\cdots {-}}})^2 \le \delta\La w^{-1}\Ra^2_{J_{\pm\pm\cdots}}\,.
$$
 At this moment we use the property 5 of $B$ called obstacle condition. 
 On intervals $I\in \cI$
the obstacle condition will provide us with $\cB(\La w\Ra_{I}, \La w^{-1}\Ra_{I}, \la_{I}) =\La w\Ra_I$.
So on a certain large  finite step $N$  we get from the iteration of \eqref{main-wT} $N$ times the following estimate
$$
\sum_{I\in D_N(J): I\in \cI} |I| \La w\Ra_I \le \frac{A\,Q\La w^{-1}\Ra_J}{\la}|J|\,.
$$
Therefore, we proved
$$
\frac1{|J|} w\{x\in J: \sum_{I\in D(J)} |\Delta_I w^{-1}|^2 \chi_I(x) >\lambda\} \le A\, Q \frac{\La w^{-1}\Ra_J}{\la}\,,
$$
which is \eqref{weakTSqF}.

\end{proof}

\bigskip

\subsubsection{Formula for the function $\cB$.\,  Monge--Amp\`ere equation with a drift.}
\label{formula-weakT}

Here is the formula for $\cB$ that satisfies all the properties in \ref{prop-weakT} (except for the last one, the boundary condition):
\begin{equation}
\label{form-weakT}
\begin{aligned}
& \cB(u, v,\la) =  \frac1{\sqrt\la} \Theta(u\sqrt\la, \frac{v}{\sqrt\la})\,,
\\
&\text{where}\,\,\, \Theta(\gamma, \tau):= \min \left(\gamma, Q e^{-\tau^2/2}\int_0^\tau e^{s^2/2} ds\right)\,.
\end{aligned}
\end{equation}

Notice that the fact that $\cB$ has the form $\cB(u, v, \la) = \frac1{\sqrt\la} \Theta(u\sqrt\la, \frac{v}{\sqrt\la})$ is trivial, this follows from property 4 called homogeneity.

Notice also that function $\Theta$ is  given in the domain enclosed by two hyperbolas
$$
H:=\{(\gamma, \tau)>0: 1\le \gamma \tau \le Q\}\,.
$$

All properties  listed at the beginning of Section \ref{prop-weakT} (except for the sixth bullet, which is boundary condition, but we do not use it anywhere) follow by direct computation.
In the next section we explain how to get this formula.

\subsubsection{Explanation of how to find such a function $\Theta$.}
\label{expl-weakT}

\bigskip

The main inequality (with $c=\frac18$) in terms of $\Theta$ becomes a ``drift concavity condition":
\begin{equation}
\label{Theta-MI}
\begin{aligned}
& \frac1{\sqrt{1+\frac{(\Delta\tau)^2}{8}}}\Theta \bigg( \sqrt{1+\frac{(\Delta\tau)^2}{8}} \frac{\gamma_-+\gamma_+}{2}, 
\frac1{\sqrt{1+\frac{(\Delta\tau)^2}{8}}} \frac{\tau_-+\tau_+}{2}\bigg)\ge 
\\
& \frac{\Theta(\gamma_-, \tau_-) + \Theta(\gamma_+, \tau_+)}{2}\,,
\end{aligned}
\end{equation}
where $(\gamma_-, \tau_-), (\gamma_+, \tau_+) \in H$, $0<\tau_-<\tau_+, \Delta \tau:=\tau_+-\tau_-$.

Assuming that $\Theta$ is smooth (we will find a smooth function), the infinitesemal version appears,  it is a sort of Monge--Amp\`ere relationship with a drift. Namely, the following matrix relationship must hold
\begin{equation}
\label{negative}
\begin{bmatrix}
\Theta_{\gamma\gamma}, & \Theta_{\gamma\tau}\\
\Theta_{\gamma\tau}, & \Theta_{\tau\tau} +\Theta + \tau\Theta_\tau -\gamma\Theta_\gamma
\end{bmatrix} \le 0\,.
\end{equation}
A direct calculation shows that this property is equivalent to the following one.
On any  curve  $\gamma =\phi(\tau)$ lying in the domain $H$ 
\begin{equation}
\label{change-var}
\begin{aligned}
&\text{and such that }\,\, \phi'' + \tau\phi' +\phi=0 
\\
& \text{we have}\, \, (\Theta(\phi(\tau), \tau))'' + \tau  (\Theta(\phi(\tau), \tau))' + \Theta(\phi(\tau), \tau)\le 0\,.
\end{aligned}
\end{equation}

This hints at a possibility to have  a change of variables $(\gamma, \tau)\to (\Gamma, T)$ such that
condition \eqref{negative} transforms to a simple concavity. To some extent this is what happens. Namely, notice the following simple
\begin{lemma}
\label{l-change-var}
Consider the following change of variable: $T=\int_0^\tau e^{s^2/2} ds$.  Then $\phi''(\tau)+\tau \phi'(\tau) +\phi(\tau)\le 0$ if and only if  $(e^{\tau^2/2} \phi(\tau))_{TT} \le 0$ and  $\phi''(\tau)+\tau \phi'(\tau) +\phi(\tau)= 0$ if and only if  $(e^{\tau^2/2} \phi(\tau))_{TT} = 0$.
\end{lemma}

\begin{proof}
The proof is a direct differentiation.
\end{proof}

This Lemma hints that the right change of variable should look like

\begin{equation}
\label{GaT}
\begin{cases} \Gamma:= \gamma e^{\tau^2/2},\\ T= \int_0^\tau e^{s^2/2} ds\,,\, (\gamma, \tau)\in \bR_+^1\times \bR_+^1\end{cases}
\end{equation}
then in the new coordinates the family of  curves $\gamma =\phi(\tau)$ such that $\phi'' + \tau\phi' +\phi=0 $ becomes a family of all  straight lines $\Gamma = CT +D$. (Notice that both families depend on two arbitrary constants.)

Denote
$$
O:= \{(\Gamma, T): (\gamma, \tau)\in G\}.
$$

The condition $ (\Theta(\phi(\tau), \tau))'' +\tau  (\Theta(\phi(\tau), \tau))' + \Theta(\phi(\tau), \tau)\le 0$ on any of these curves  becomes
\begin{equation}
\label{TT}
\left(e^{\tau^2/2}\Theta(\phi(\tau), \tau)\right)_{TT} \le 0\,, \,\, \Gamma=CT+D\,,\,\, (\gamma, \tau)\in G\,.
\end{equation}
which is the concavity of $e^{\tau^2/2}\Theta (\gamma, \tau)$  in a new coordinate $T$ along the line $\Gamma = CT +D$.
Let us rewrite two functions in the new coordinates:
$$
\Phi(\Gamma, T):= \Theta(\gamma, \tau)\,,\,\, U(T) := e^{\tau^2/2}\,.
$$
Then \eqref{TT} transforms into
\begin{equation}
\label{concUPhi}
\forall C, D \in \bR\,,\,\,\left(U(T)\Phi(CT+D, T))\right)_{TT} \le 0\,, \,\, (\Gamma, T)\in O\,.
\end{equation}
This is just a concavity of $U(T) \Phi(\Gamma, T)$ on $O$ of course. Notice that neither $H$ nor $O$ are convex, so we  should understand \eqref{concUPhi} as a local concavity in $O$: just the negativity of its second differential form 
$$
d^2_{\Gamma, T} (U(T)\Phi(\Gamma, T)\le 0\,,\,\, (\Gamma, T)\in O\,.
$$

\medskip

So we reduce the question to finding a concave function in new coordinates. Now we choose a simplest possible concave function:
$$
U(T) \Phi(\Gamma, T) := \min (\Gamma, KT)\,,
$$
where the constant $K=K(Q)$ will be chosen momentarily.

\medskip

If we write down now $\Theta(\gamma, \tau)= \Phi(\Gamma, T)$ in the old coordinates, we get exactly function $\Theta$ from \eqref{form-weakT} (we need to define constant $K$ yet), namely,
\begin{equation}
\label{Theta5}
\Theta(\gamma, \tau):=\min \left(\gamma, K e^{-\tau^2/2}\int_0^\tau e^{s^2/2} ds\right)\,.
\end{equation}

Recall that now we can consider
\begin{equation}
\label{Btest}
B(u, v, \la)=\frac1{\sqrt{\la}} \Theta(u\sqrt{\la}, \frac{v}{\sqrt{\la}})
\end{equation}
and we are going to apply Theorem \ref{t-weakT} to it. But we need to choose $K$ to satisfy all the conditions (except the last one) at the beginning of Section \ref{prop-weakT}.

First of all it is now very easy to understand why the form of the domain $H=\{1\le \gamma \tau\le Q\}$ plays the role. In fact, by choosing
$$
K=AQ
$$
with some absolute constant $A$, we guarantee that in this domai n our function $\Theta$ satisfies the obstacle condition
\begin{equation}
\label{obstacle-weakT}
\Theta(\gamma, \tau)  = \gamma\,\, \text{as soon as}\,\, \tau\ge a_0>0\,,
\end{equation}
where $a_0$ is an absolute positive constant.
In fact, for all sufficiently small $\tau$, $e^{-\tau^2/2}\int_0^\tau e^{s^2/2} ds \asymp \tau$, and therefore, 
for all sufficiently small $\tau$ (smaller than a certainn absolute constant)
$$
\Theta(\gamma, \tau):=\min \left(\gamma, K \tau\right)\,.
$$
The fifth condition at the beginning of Section \ref{prop-weakT} (the obstacle condition) requires then that $\min \left(\gamma, K \tau\right)=\gamma$ if $\tau\ge a_0>0$. But on the upper hyperbola then $\gamma=Q/a_0$ for $\tau=a_0$. We  see that the smallest possible $K$ we can choose to satisfy the obstacle condition is $K\asymp Q$.

\bigskip

Secondly, function $\Theta$ satisfies the infinitesimal condition \eqref{negative} by construction. But we need to check that the main inequality \eqref{Theta-MI} is satisfied as well.

This can be done by the following lemma.

\begin{lemma}
\label{redu-weakT}
Inequality \eqref{Theta-MI} for function $\Theta$ built above holds if and only if the following inequality is satisfied for $\phi(\tau):=  e^{-\tau^2/2} \int_0^\tau e^{s^2/2} ds$:
\begin{equation}
\label{phi}
\begin{aligned}
& \frac1{\sqrt{1+\frac{(\Delta\tau)^2}{100}}}\phi \bigg(\frac{\tau_1+\tau_2}{2\sqrt{1+\frac{(\Delta\tau)^2}{100}}}\bigg) \ge 
\\
& \frac{\phi(\tau_1) +\phi(\tau_2)}{2}\,, \, \forall \,0<\tau_1\le\tau_2 \le\tau_0\,,
\end{aligned}
\end{equation} 
with some absolute positive small constant  $\tau_0$.
\end{lemma}

\begin{proof}
Lemma is easy, because we can immediately see that the main inequality \eqref{Theta-MI} commutes with the operation of minimum. 
\end{proof}

We are left to prove  \eqref{phi}. We leave this as exercise to the reader. It is a direct if rather involved calculation.

\bigskip

\subsection{A Partial Differential Inequality for the sharp weak  estimate of  weighted square function}
\label{Full}

In this Section we are going to reduce the harmonic analysis problem of sharp estimate of weak norm of weighted square function to a certain Partial Differential Inequality (PDI) with an obstacle condition in a rather simple subdomain of $\bR^3$.

We introduce the following function of $5$ real variables (compare with \eqref{supTest}):

\begin{equation}
\label{supFull}
\begin{aligned}
& \Bel_Q(F, f, u, v, \la):=
\\
& \sup \frac{1}{|J|} w\left\{x\in J: \sum_{I\in D(J)} |\Delta_I (\phi w^{-1})|^2 \chi_I(x) >\lambda\right\} \,,
\end{aligned}
\end{equation}
where supremum is taken over all $w\in A_2, [w]_{A_2}\le Q$, such that 
$$
\La w\Ra_J =u, \La w^{-1}\Ra_J = v\,,
$$
and over all functions $\phi$ such that $\La \phi w^{-1}\Ra_J = f, \La \phi^2 w^{-1}\Ra_J = F$.

{\bf  Remark.}
Our usual disclaimer: the reader should not be confused by notations, but should be warned: of course $F, f, u, v$ stand for numbers, not functions.

\bigskip

Notice that by scaling argument  function $\Bel_Q$ does not depend on $J$ but depends on $Q=[w]_{A_2}$. Notice also that numbers $F, f$ and $u, v$ are not independent, in particular,
\begin{equation}
\label{domainFull}
f^2 \le F\cdot v,\,\ 1\le u\cdot v\le Q,\, \la\ge 0\,.
\end{equation}
This describes the domain $\Omega=\Omega_Q$ of definition of function $B$. 
Let
\begin{equation}
\label{hat}
\sup_{\Omega_Q}\frac{\la B(F, f, u, v, \la)}{F} =: R\,.
\end{equation}
By simple one-point-singularity weight we know that for an absolute positive const ant $\kappa$ the following holds
$$
R\ge \kappa\,Q\,.
$$

The question is of course whether one of the following is true
\begin{equation}
\label{whichone}
\begin{cases}
\exists A<\infty \, \,\,R\le A Q\\
\limsup_{Q\to \infty} R/ Q =\infty\\
\exists \eps>0, a_\eps>0: \, R \ge a_\eps Q (\log Q)^\eps\,?\end{cases}
\end{equation}

From \cite{LaSc}, \cite{DSaLaRey} we know (recall that $\hat Q$ is not the weak norm but is the square of the weak  norm $\|S:L^2(w)\to L^{2, \infty} (w)\|)$:
\begin{equation}
\label{aboveFull}
R \le A Q \log Q\,.
\end{equation}

The goal of this section is to reduce the estimate of $R$ to a solution of a PDI in a certain simple  domain in $\bR^3$.

\begin{lemma}
\label{even}
Function $\Bel_Q$ introduced above is 1) even in variable $f$, 2) decreases in variable $\la$, 3) it is not greater than $u$.
\end{lemma}

\begin{proof}
All claims of this lemma follow immediately from the definition of $\Bel_Q$.
\end{proof}

\subsection{The main inequality and the  Partial Differential version of it.}
\label{mainSqfull}
Exactly like in Section \ref{prop-weakT} we can write
the following {\it main inequality}. If $P=(F, f, u,v, \la), P_+=(F_+, f_+, u_+,v_+, \la_+), P_-=(F_-, f_-, u_-,v_-,\la_-)$ belong to $\Omega_Q$, and
$F=\frac12(F_+ +F_-)$, $f=\frac12(f_+ +f_-)$, $u=\frac12(u_+ +u_-), v= \frac12(v_+ + v_-), \la=\min(\la_+, \la_-)$, then {\it the main inequality} holds with constant $\kappa=1$:
\begin{equation}
\label{MIfull}
\Bel\left(F, f, u, v, \la+\kappa (f_+-f_-)^2\right)-\frac{\Bel(P_+)+ \Bel(P_-)}{2} \ge 0.
\end{equation}

In particular, function $\Bel$ is convex in $(F, f, u, v)$ variables and monotone decreasing (and left continuous) in $\la$.
In infinitesimal form this inequality becomes
\begin{equation}
\label{MIfull_inf}
-\frac12 d^2_{F, f, u, v}\Bel  + \kappa \frac{\partial \Bel}{\partial{\la}} \cdot (df)^2\ge 0\,.
\end{equation}
By Aleksandrov's theorem \cite{EG} we have this inequality pointwise a. e. in $\Om$, if by $ \frac{\partial \Bel}{\partial{\la}}$ we denote upper derivative in $\la$.

\bigskip

\subsection{Reduction of the number of variables. Properties of $\Bel=\Bel_Q$ and $\Psi=\Psi_Q$.}
\label{neity}
By the construction of $\Bel=\Bel_Q$ we can easily see two homogeneity properties:
\begin{align*}
& \Bel(\frac{F}{t}, \frac{f}{t}, tu, \frac{v}{t}, \frac{\la}{t^2}) = t \Bel(F, f, u, v, \la)\,,
\\
& \Bel(\frac{F}{s^2}, \frac{f}{s}, u, v, \frac{\la}{s^2}) = \Bel(F, f, u, v, \la)\,.
\end{align*}
Denote $\Phi:=\Phi_Q(\cdot, \cdot, \cdot, \cdot):= \Bel(\cdot, \cdot, \cdot, \cdot), 1)$, then from these relationships we obtain
\begin{align*}
& \Phi(\frac{F}{\sqrt{\la}}, \frac{f}{\sqrt{\la}}, \sqrt{\la} u, \frac{v}{\sqrt{\la}}) = \sqrt{\la}  \Bel(F, f, u, v, \la),,
\\
& \Phi(\frac{F}{\la}, \frac{f}{\sqrt{\la}}, u, v) = \Bel(F, f, u, v, \la)\,.
\end{align*}
Hence, with obvious notations for $A, c$, $\Phi(\sqrt{\la}A, c, \sqrt{\la}u, \frac{v}{\sqrt{\la}}) = \sqrt{\la} \Phi (A, c, u, v)$. Denote
by $\Psi:=\Psi_Q(\cdot, \cdot, *) = \Phi(\cdot, *, \cdot, 1)$. Then denoting by 
\begin{equation}
\label{changevar}
a= \sqrt{\la} A=\frac{Fv}{\la}, \,\,b=uv, \,\,c=\frac{f}{\sqrt{\la}}
\end{equation}
 we can write
\begin{equation}
\label{toPsi}
\Psi(a, b, c) =v \Phi(\frac{a}{v}, c, u, v)=v \Bel(F, f, u, v, \la)\,.
\end{equation}
We can rewrite all information on $\Bel$ in terms of $\Psi$.
For example, the domain $\Om_Q=\{(F, f, u, v, \la): f^2 \le Fv, 1< uv<Q, \la>0\}$  of definition of $\Bel$ is mapped to
the domain of definition of $\Psi$:
$$
G:=G_Q:=\{(a, b, c): c^2 \le a, 1< b< Q\}\,.
$$
\eqref{hat} becomes
\begin{equation}
\label{hatPsi}
\sup_{(a, b, c)\in G_Q}\frac{\Psi(a, b, c)}{a} =R\,.
\end{equation}

By the formula
\begin{equation}
\label{PsiBel}
\Psi(\frac{vF}{\la}, uv, \frac{f}{\sqrt{\la}})= v \Bel (F, f, u, v, \la)
\end{equation}
we see that $\Psi$ is concave in all its three variables (just vary $(F, f, u)$ to see this).

The main property \ref{MIfull_inf} becomes the following a. e. pointwise in $G$ property (it is a direct calculation with the change of variables \eqref{changevar}):

\begin{equation}
\label{Th_form}
\begin{bmatrix} &\Psi_{aa},\,\,\,\,\, &\Psi_{ab}\,\,\,&\Psi_{ac}\,\, &0\\
&\Psi_{ab}\,\,&\Psi_{bb}\,\,&\Psi_{bc}\,\, &0\\
&\Psi_{ac}\,\,\,\,\,&\Psi_{bc}\,\,\, &\Psi_{cc} +\kappa (c\Psi_c+ 2a\Psi_a)\,\,&-\Psi_c\\
&0\,\,\,\,\,\,&0\,\,\, &-\Psi_c\,\, & 2\Psi-2a\Psi_a-2b\Psi_b
\end{bmatrix}\le 0
\end{equation}
in $G=G_Q$ and with $\kappa=1$.

We also have obstacle condition exactly as in Section \ref{prop-weakT}. By the same example of the same weight $w_s$ as in Section \ref{prop-weakT} we can see
that
$$
\Bel(t^2v, tv, u, v, \la) = u\,,\,\,\text{as soon as}\,\, \la \le \delta v^2
$$
with a certain small absolute positive $\delta$ and any (positive or negative)  $t$ as soon as $uv\ge 10$. In therms of $\Psi$ this reads as follows:
$$
\Psi(c^2, b, c) = b\,,\,\,\text{as soon as}\,\,  \delta^{-1}< c, b\ge 10\,.
$$
We have, thus, an obstacle condition on the intersection of the parabolic boundary 
\begin{equation}
\label{Gparab}
\pd_{parab} G=\{(a, b, c): a=c^2, 1\le b\le Q\}\,.
\end{equation}
with $b\ge 10$.
 But function $\Psi$ is concave in its three variables, therefore
\begin{equation}
\label{obstPsi}
\Psi(a, b, c) = b\,,\,\,\text{as soon as}\,\,  \delta^{-1}< c, b\ge 10\,.
\end{equation}
We used here the fact that by  \eqref{toPsi}
$$
\Psi(a, b, c) \le b\,.
$$

Function of $(F, f, u, v, \la)$ in the left hand side of \eqref{PsiBel} is absolutely continuous in each of its variables (we reasoned that $\Psi$ is concave), so then is the function in the right hand side of \eqref{PsiBel}.  Let us differentiate this formula in $\la$:
$$
-\frac{\pd \Psi}{\pd a} \frac{vF}{\la^2} -\frac12\frac{\pd \Psi}{\pd c} \frac{f}{\la^{3/2}} \le 0\,,
$$
because $\Bel$ decreases in $\la$.
We obtained one more property of $\Psi$:

\begin{equation}
\label{ac-deriv}
2a\Psi_a +c\Psi_c \ge 0\,.
\end{equation}

From \eqref{Th_form}  one see that 
\begin{equation}
\label{ab-deriv}
a\Psi_a +b\Psi_b \ge \Psi\,.
\end{equation}
This implies that function $\phi(t) := \Psi (ta, tb, c)$ defined on $t\in [\max(\frac{c^2}{a},\frac1{b}), \frac{Q}{b}]$ is such that $\frac{\phi(t)}{t}$ is increasing.
But $\phi$ is convex. So
\begin{equation}
\label{phi-est}
 Rat -\Psi (ta, tb, c) =Rat-\phi(t) \le  Ra(\frac{c^2}{a} + \frac1{b}) = R( c^2+ \frac{a}{b})\,.
 \end{equation}

 \subsection{The boundary conditions of $\Bel_Q$ and $\Psi_Q$.}
 \label{boundaryPsi}
 
 Let us consider $\Bel_Q(F, f, u, v, \la)$ on $F=v, f=v$. This corresponds to the testing function $\vf$ to be just $\1_J$.   We obtain exactly the function $B_Q$ from Section \ref{prop-weakT}:
 $$
 B_Q(u, v, \la)= \Bel_Q(v, v, u, v, \la)\,.
 $$
 
 By homogeneity we can write it as
 $$
B_Q(u, v, \la) = :\frac{1}{\sqrt{\la}} \Theta_Q(u\sqrt{\la}, \frac{v}{\sqrt{\la}})\,,
$$
where this equation serves as the definition of $\Theta_Q$. By \eqref{PsiBel} we have a chain of equalities:
$$
\frac{1}{\sqrt{\la}} \Theta_Q(u\sqrt{\la}, \frac{v}{\sqrt{\la}})= B_Q(v, v, u, v, \la)= \frac1{v} \Psi_Q (\frac{v^2}{\la}, uv, \frac{v}{\sqrt{\la}})\,.
$$
Thus, on the boundary $G_{parab}$ from \eqref{Gparab} we obtain
\begin{equation}
\label{onGparab}
\Psi_Q(c^2, b, c) = c \Theta_Q(\frac{b}{c}, c)\,.
\end{equation}

Recall that we found the formula for certain functionn $\Theta$ in \eqref{form-weakT}.
And we actually proved that 
\begin{equation}
\label{Thety}
\Theta_Q \le  A\Theta
\end{equation}
in the domain $H=\{(x, y)\in \bR_+^2: 1< xy<Q\}$.
But by formula \eqref{form-weakT} $\Theta(x, y) \le Qy$ at least for $y\le Q^{-1/2}$.

So for small $c, c\le Q^{-1/2}$ we have
$$
\Psi_Q(c^2, b, c) \le AQc^2\,.
$$
And this is true  for all $b\in (1, Q)$.

But  there is another obstacle condition that shows that
$$
\Psi_Q(c^2, b, 0) \ge A'Qc^2\,.
$$

%

\subsection{Full weak weighted estimate of square function}
\label{fullweakSqF}

First we provide information on full weak estimate for square function, namely on 
the asymptotics of  constant $C(\QQ)$ from \eqref{weakSqF}.

The following result proved by C. Domingo-Salazar, M. Lacey, G. Rey  \cite{DSaLaRey}.
\begin{theorem}
\label{fullweakSqF_th}
Let $w\in A_2$, then the norm of the square function operator 
$S\colon L^2(w)\to L^{2, \infty} (w)$ is bounded by $C [w]_{A_2}^{1/2} \log^{1/2}(1+ [w]_{A_\infty})$, where $C$ is an absolute constant.
\end{theorem}

The claim of the theorem is equivalent to the following inequality that should be proved for an arbitrary function $\vf \in L^2((0,1), w)$:
\begin{equation}
\label{fullweakSq_ineq}
w\{x\in (0,1)\colon S\vf (x) >1\} \le C  [w]_{A_2}\log(1+ [w]_{A_\infty})\, \|\vf\|_{w}^2\,.
\end{equation}

Let us make several remarks. 

{\bf  Remark.}
In the previous section we studied the case of special functions $\vf$, namely, we proved that if $\vf= w^{-1}\1_{I}$, $I\in \cD$, then 
inequality \eqref{fullweakSq_ineq} can be strengthened. By this we mean that for such special test functions the estimate above
has the right hand side $C  [w]_{A_2}\, \|\vf\|_{w}^2$. There is no logarithmic blow-up.

{\bf  Remark.}
Exponent $p=2$ is critical for Theorem \ref{fullweakSqF_th}. By this we mean 
that one can quite easily deduce  from this theorem the result for $p>2$: 
if $w\in A_p$, then the norm of the square function operator 
$S\colon L^p(w)\to L^{p, \infty} (w)$ is bounded by $C [w]_{A_p}^{1/2} \log^{1/2}(1+ [w]_{A_\infty})$, 
where $C$ is an absolute constant. For that reduction to the case $p=2$ he reader can look at \cite{DSaLaRey}.
It is important to note that for $1\le p<2$, and $w\in A_p$, \cite{DSaLaRey} 
proves the estimate $C [w]_{A_p}^{1/2}$ for the weak norm of the square function operator.


\section{Restricted weak weighted estimate of the square function} 
\label{restrictedSqF}

\subsubsection{Sparse square function operators}

{\bf Definition.}
A family $\mathcal{S}$ of intervals of $\cD$ is called $\eps$-sparse if the following condition is satisfied
\begin{equation}
\label{sp-def}
\sum_{I\in \mathcal{S}, \atop I\subsetneq J} |I| \le \eps |J|,\quad \forall J\in \mathcal{S}\,.
\end{equation}

{\bf Definition.}
\label{sp-op}
Sparse square function operator is defined for each sparse family $\mathcal{S}$ as follows
$$
S^{\sp}\vf \df S^{\sp}_{\mathcal{S}}\vf \df \big(\sum_{I \in \mathcal{S}} \La \vf \Ra_I^2 \1_I\big)^{1/2}\,.
$$

\begin{theorem}
\label{toSqSp}
For any $\eps>0$ and any $\vf \in L^1$ there exists a constant $C=C(\eps)$ 
independent of $\vf$ and a sparse family $\mathcal{S}$ (depending on $\eps$ and on $\vf$) such that pointwisely almost everywhere
$$
S\vf \le C S^{\sp}\vf\,.
$$
\end{theorem}

\begin{proof}
It is well known (see e.\,g. \cite{Wilson}) that the square function operator is weakly bounded in unweighted $L^1$.  
Let us call $A$ the norm of the operator $S$ from $L^1$ to $L^{1,\infty}$. Fix $\eps$ and let $C=100A/\eps$.
 We start with interval $I_0=(0,1)$, put $\mathcal{S}_0\df \{I_0\}$,  and define the first generation of stopping intervals $\mathcal{S}_1$ as follows:
 $Q\in \mathcal{S}_1$ if  it is the maximal interval in $I_0$ such that
 $$
 S_Q^{I_0}\vf \df \big(\!\!\sum_{I \in \cD, \atop Q\subset I \subset I_0} (\Delta_I \vf)^2\big)^{1/2}  > C\La |\vf| \Ra_{I_0}\,;
 $$
 
The second generations of stopping intervals $\mathcal{S}_2$ will be nested inside the first generation $\mathcal{S}_1$. 
For every $I\in \mathcal{S}_1$ we define its subintervals from $\cD$ by the same rule as before, 
but with $I$ playing the r\^ole of   $I_0$. Namely, we define the first generation of stopping intervals $\mathcal{S}_2$ inside $I\in \mathcal{S}_1$  as follows:
 $Q\in \mathcal{S}_2$ if it is the maximal interval in $I$ such that
 $$
 S_Q^{I}\vf \df \big(\!\!\sum_{J \in \cD, \atop Q\subset J \subset I} (\Delta_J \vf)^2\big)^{1/2}  > C\La |\vf| \Ra_{I}\,;
 $$
 
 We continue the construction of generations of intervals $\mathcal{S}_3, \mathcal{S}_4, \dots$ recursevely, and we put $\mathcal{S}\df \cup_{k=0}^\infty \mathcal{S}_k$.

 Notice that by the fact that operator $S$ and dyadic maximal operator $M$ are  
 weakly bounded in unweighted $L^1$ and  from our choice of constant $C$ at the 
 beginning of the proof, we get that
  $ \sum_{Q\in \mathcal{S}_1} |Q| \le \frac{\eps}{50}|I_0|$, and similarly,
  $$
  \sum_{Q\in \mathcal{S}_{k+1}, \atop Q\subset I} |Q| \le \frac{\eps}{50}|I|,\quad \forall I\in \mathcal{S}_k\,.
  $$
  
  Obviously, and with a good margin, we obtained that  $\mathcal{S}$ is $\eps$-sparse.
  
  Now to see the pointwise estimate of the theorem, let us notice that given $x\in I_0$, 
  which is not an end-point of any dyadic interval, we will be able to find the tower 
  of intervals $\dots \subsetneq I_{k}\subsetneq I_1\subsetneq I_0$ such that $x$ is contained in all of them and such that
  $I_k\in \mathcal{S}_k$. This tower may degenerate to just one interval $I_0$, or it can be an infinite tower. 
  But the set of points for which the tower is infinite has Lebesgue measure zero. This is clear from the fact that $\mathcal{S}$ is sparse.
  In any case,
  \begin{align*}
  S^2\vf(x) &= \sum_{I\in \cD, \atop I_1\subsetneq I \subset I_0} (\Delta_I\vf)^2 
    + \sum_{I\in \cD, \atop I_2\subsetneq I \subset I_1} (\Delta_I\vf)^2 
    \\
    &+\sum_{I\in \cD, \atop I_3\subsetneq I \subset I_2} (\Delta_I\vf)^2   
    +\dots
  \end{align*}
  
  But then, using our stopping criterion we see that  the last expression is bounded by
  $$
  C^2\big( \La |\vf|\Ra^2_{I_0} + \La |\vf|\Ra^2_{I_1} + \La |\vf|\Ra^2_{I_2}+\dots\big),
  $$
  which proves the theorem.

\end{proof}

{\bf  Remark.}
Now we will see what can be changed in the reasoning  of  \cite{DSaLaRey} in order 
to obtain the estimate better than \eqref{final-w-SqF} for a special choice of $\vf =w^{-1}$. 
In fact, the Bellman function technique of the previous section gave us a better estimate in this particular case $\vf=w^{-1}$:
\begin{align}
\label{final-w-SqF-test}
w\big\{x\in I_0 \colon S^{\sp}w^{-1} >3\big\} \le  A\QQ \int_{I_0} w^{-1}\, dx\,.
\end{align}

\bigskip

\begin{theorem}
\label{restrictedSqFsparse_thm}
Let $w\in A_2$ and $\mathcal{S}$ be a collection of sparse dyadic in servals in $\cD(I_0)$.  Let $\mathcal{S}^{\sp}$ be the sparse square function operator built on this collection. Then the restricted weak type of the operator $\mathcal{S}^{\sp}_{w^{-1}}$ from $L^2(w^{-1})$ to $L^{2, \infty}(w)$ is bounded by $A\QQ^{1/2}$, where $A$ is an absolute constant.
\end{theorem}

We need the following well-known result:

\begin{lemma}
\label{weak-w-MAX}
Let $M$ be the dyadic maximal operator, and $w$ be in dyadic $A_2$. 
Then the norm of $M$ from $L^2(w)$ to $L^{2, \infty}(w)$ is bounded by $\QQ^{1/2}$.
\end{lemma}

\begin{proof}
Given a test function $\vf \ge 0$, let $\{I\}$ be the maximal dyadic intervals for which $M\vf >1$. Then
\begin{align*}
&\sum_I w(I) \le \int \sum \frac{w(I)}{|I|} \vf \1_I \, dx =  \int \sum_I \frac{w(I)}{|I|} \vf \1_I \, w^{-1/2} w^{1/2}  dx
\\
&\le \Big(\sum_I  \Big(\frac{w(I)}{|I|} \Big)^2 w^{-1}(I)\Big)^{1/2} \|f\|_w \le  \QQ^{1/2} \big( \sum_I w(I)\big)^{1/2} \|\vf\|_w
\end{align*}

Hence,
$$
\big(w\big\{x\colon M\vf >1\big\}\big)^{1/2} =(\sum_I w(I))^{1/2} \le \QQ^{1/2} \|\vf\|_w\,,
$$
which is precisely what the lemma claims.
\end{proof}

{\bf  Remark.}
One can skip the word ``dyadic" everywhere in the statement of this lemma. 
Also one can generalize the  statement to $\bR^n$. Then lemma remains trues, 
only the estimate becomes by $C_n \QQ^{1/2}$.  See \cite{Wilson}.

Let us  assume that $0\le \vf \le w^{-1} \1_{I_0}$. 

As in  \cite{DSaLaRey} let $\mathcal{S}$ be an $\eps$-sparse system of dyadic sub-intervals of $I_0$, and $S^{\sp}$ be a corresponding sparse square function.
 We split $\mathcal{S}= \mathcal{S}^0 \cup\cup_{m=0}^\infty \mathcal{S}_m$, where
 $\mathcal{S}^0$ consists of intervals of $\mathcal{S}$ such that $\La \vf \Ra_I >1$, and $\mathcal{S}_{m+1}$ consists of intervals of $\mathcal{S}$ such that 
 \begin{equation}
 \label{2m}
 2^{-m-1}  < \La |\vf| \Ra_I \le 2^{-m},\quad m=0,1,\dots\,.
 \end{equation}
 We denote 
 $$
 S_m^{\sp}\vf \df \big(\sum_{I\in \mathcal{S}_m} \La \vf \Ra_I^2 \1_I\big)^{1/2}\,.
 $$
 We are going to estimate these measures
 \begin{align*}
& W_0\df w\big\{x\in I_0\colon (S_0^{\sp}\vf)^2>1\big\},\,\, 
\\
& W\df w\big\{x\in I_0\colon \sum_{m=0}^{\infty} (S_{m}^{\sp}\vf)^2>4\big\}\,.
 \end{align*}

 The estimate of $W_0$ is easy. The sum $(S_0^{\sp}\vf)^2$ is supported on intervals where the 
 dyadic maximal function of $\vf$ is bigger than $1$. The set, 
 where the dyadic maximal function is bigger than $1$ has $w$-measure bounded by
 $\QQ \|\vf\|_w^2$ by Lemma \ref{weak-w-MAX}.
 
 \bigskip
 
Now we work with $W$, and we start exactly as in the estimate of $W$ above. Namely,
 the support of $S^{\sp}_m \vf$ is in $\cup Q^{(m)}_i$, where 
$Q^{(m)}_i\in  \cS_m$ that  are the maximal dyadic intervals  $Q$
with the property
\begin{equation}
\label{prop} 
 \sum_{I\in \cS_m, \,Q\subset I \subset I_0} \La |\vf|\Ra_I^2 \1_I \ge 2 ^{-m}
\end{equation}
By  maximality  and by \eqref{2m} 
\begin{equation}
\label{above}
 \sum_{I\in \cS_m, \,Q^{(m)}_i\subset I \subset I_0} \La |\vf|\Ra_I^2 \1_I \le 2 ^{-m} + 2^{-2m}\,.
 \end{equation}


To estimate $W$ we use  a   union bound.
\begin{align*}
&w\big\{x\in I_0\colon \sum_{m=0}^{\infty} (S_{m}^{\sp}\vf)^2>4\big\} 
\\
&\le w\big\{x\in I_0\colon \sum_{m=0}^{\infty} (S_{m}^{\sp}\vf)^2>\sum_{m=0}^\infty 2 ^{-m/2}\big\}
\\
& \le \sum_{m=0}^\infty w\big\{x\in \cup_i Q^{(m)}_i: (S_{m}^{\sp}\vf)^2> 2 ^{-m/2}\big\}=
\\
& \le \sum_{m=0}^\infty\sum_i w\big\{x\in Q^{(m)}_i: (S_{m}^{\sp}\vf)^2> 2 ^{-m/2}\big\}\,.
\end{align*}

Now let us fix $Q^{(m)}_i$. Then
$$
\sum_{I\in \cS_m, \, Q^{(m)}_i \subset I\subset I_0} \La |\vf|\Ra_I^2 \,\1_I \ge 2^{-m}
$$
On this interval $ \La |\vf|\Ra_I^2 \le 2^{2m}$ by the definition of $\cS_m$. Hence,
$$
\sum_{I\in \cS_m, \, Q^{(m)}_i \subset I\subset I_0} \1_I \ge 2^{m}\,.
$$
Therefore, every $Q^{(m)}_i$ lies at least $2^{m}$ generations down in $\cS$. By sparse  choice
\begin{equation}
\label{down}
\big|\cup_i Q^{(m)}_i\big|\le e^{-a\,2^{m}}\,.
\end{equation}

\bigskip

Consider maximal dyadic intervals $\ell \in \mathcal{S}_m$, $\ell\subset Q^{(m)}_i$, on which $(S_{m}^{\sp}\vf)^2> 2 ^{-m/2}$.  Call them $\cL(Q^{(m)}_i)$.
Let us estimate $\sum_{\ell\in \cL(Q^{(m)}_i)} |\ell|$.

Notice that for each  $\ell \in \cL(Q^{(m)}_i)$
$$
\sum_{I\in \cS_m, \,  \ell\subset I \subset Q^{(m)}_i} \La |\vf|\Ra_I^2\1_I \ge 2^{-m/2} - 2^{-m}- 2^{-2m} \ge \frac12 2^{-m/2}\,.
$$
This is by \eqref{above}. Now we use again that $ \La |\vf|\Ra_I^2\le 2^{2m}$ for all terms of this sum. Hence,
$$
\sum_{I\in \cS_m, \,  \ell\subset I \subset Q^{(m)}_i} \1_I \ge 2^{-m/2} - 2^{-m}- 2^{-2m} \ge  2^{-m}\,.
$$
Therefore, every $\ell$ from  $\cL(Q^{(m)}_i)$  lies at least $2^{m}$ generations down from $Q^{(m)}_i$ in $\cS_m$, and, thus, in $\cS$. By sparse construction
\begin{equation}
\label{cL}
\sum_{\ell\in \cL(Q^{(m)}_i)} |\ell| \le e^{-a 2^m} |Q^{(m)}_i|\,.
\end{equation}

\bigskip

Recall that we assumed 
$$
\vf \le w^{-1},
$$
Hence, 
$$
\La w^{-1}\Ra_\ell^{-1} \le \La \vf \Ra_\ell^{-1}.
$$
By definition $\La \vf\Ra_\ell \ge 2^{-m-1}$, hence 
\begin{align*}
&\La w\Ra_\ell \le \QQ\La w^{-1}\Ra_\ell^{-1} \le \QQ \La \vf\Ra_\ell^{-1}
\\
& \le 2^{m+1} \QQ , \quad \forall \ell \in \cL(Q^{(m)}_i).
\end{align*}

Therefore, by \eqref{cL} we obtain
\begin{align*}
& w\big\{x\in Q^{(m)}_i\colon  (S_{m}^{\sp}\vf)^2> 2 ^{-m/2}\big\} =\sum_{\ell \in   \cL(Q^{(m)}_i)} w(\ell) 
\\
& = \sum_{\ell \in \cL(Q^{(m)}_i)} \La w\Ra_\ell |\ell| \le 2^{m+1}\QQ  
\sum_{\ell \in \cL(Q^{(m)}_i)} |\ell | =  2^{m+1}\QQ\, e^{-a 2^m} \big| Q^{(m)}_i\big|.
\end{align*}

\medskip

We can conclude now that
$$
w\big\{x\in Q^{(m)}_i\colon  (S_{m}^{\sp}\vf)^2> 2 ^{-m/2}\big\}  \le 2^{m+1}\QQ\, e^{-a 2^m}  |Q^{(m)}_i|\le  
$$
$$
2^{2m+2}  \QQ  \, e^{-a 2^m}\,\La \vf\Ra_{Q^{(m)}_i}|Q^{(m)}_i|=
$$
$$
2^{2m+2}  \QQ  \, e^{-a 2^m}\,\int_{Q^{(m)}_i} \vf \, dx,
$$
where the last inequality follows from the definition of $\mathcal{S}_m$ and the fact that $Q^{(m)}_i\in \mathcal{S}_m$.

\medskip

Gathering things together, we obtain \eqref{final-w-SqF-test}:
\begin{align*}
&   w\big\{x\in I_0\colon \sum_{m=0}^{\infty} (S_{m}^{\sp}\vf)^2>4\big\}  \le
\\
&  \sum_{m=0}^\infty \sum_{i} w\big\{x\in Q^{(m)}_i\colon  (S_{m}^{\sp}\vf)^2> 2 ^{-m/2}\big\} 
\\
& \le \QQ \sum_{m=0}^\infty \sum_{i}  2^{2m+2} e^{-a 2^m} \int_{Q^{(m)}_i} \vf\, dx
\\
&\le \QQ \sum_{m=0}^\infty 2^{2m+2} e^{-a 2^m} \int_{\cup Q^{(m)}_i} \vf\, dx 
\\
& \le  \QQ \int_{I_0} \vf \, dx  \sum_{m=0}^\infty 2^{2m+2} e^{-a 2^m} \le A \QQ \int_{I_0} \vf\, dx \,.
\end{align*}

\medskip

Estimate
\begin{equation}
\label{vf}
 w\big\{x\in Q^{(m)}_i\colon \sum_{m=0}^{\infty} (S_{m}^{\sp}\vf)^2>4\big\} \le A \QQ \int_{I_0} \vf\, dx 
\end{equation}
was just obtained by using one assumption: $\vf \le w^{-1}$. We want to exchange in the right hand side of \eqref{vf} the integral $\int_{I_0} \vf\, dx $ for the integral $\int_{I_0} \vf^2 w\, dx $. This is trivial if one more property of $\vf$ holds, namely,
if pointwisely
$$
\vf \le C \vf^2 w\,.
$$
This is compatible with $\vf \le w^{-1}$ if for a.\,e. point $x\in I_0$ one of the following properties occurs: for every $x\in I_0$ either
1)   $\frac1{C} w^{-1}(x) \le \vf(x) \le w^{-1}(x)$, or 2) $\vf(x)=0$.

We conclude that \eqref{vf} holds for every $\vf$ of the form $\vf = w^{-1} \1_{E}$, where $E$ is a measurable subset of $I_0$.

In particular, Theorem \ref{restrictedSqFsparse_thm} is proved.

Now we can use sparse domination Theorem \ref{toSqSp}. Then this theorem gives us the following one.

\begin{theorem}
\label{restrictedSqF_thm}
Let $w\in A_2$ and $\mathcal{S}$.  Let $S$ be the dyadic square function operator. Then the restricted weak type of the operator $S_{w^{-1}}$ from $L^2(w^{-1})$ to $L^{2, \infty}(w)$ is bounded by $A\QQ^{1/2}$, where $A$ is an absolute constant.
\end{theorem}

\end{document}